\documentclass{amsart}
\usepackage{amssymb}
\usepackage{amsmath,amssymb,hyperref,mathrsfs,color}
\usepackage{paralist}
\usepackage[all,cmtip]{xy}
\usepackage{calc}
 \newtheorem{TheoremX}{Theorem}
 \newtheorem{The}{Theorem}[section]
 
 \newtheorem{Lem}[The]{Lemma}
 \newtheorem{Pro}[The]{Proposition}
 \theoremstyle{definition}
 \newtheorem{Def}[The]{Definition}
 \theoremstyle{remark}
 \newtheorem{Rem}[The]{Remark}
 
 \numberwithin{equation}{section}
\newcommand{\T}{\mathbb{T}}
\newcommand{\R}{\mathbb{R}}
\newcommand{\Z}{\mathbb{Z}}
\newcommand{\N}{\mathbb{N}}
\begin{document}
\title[]{On class A Lorentzian 2-tori with poles II:\\ Foliations by timelike lines}

\author{Liang Jin, Lu Peng \and Xiaojun Cui}
\address{Department of Mathematics, Nanjing University of Science and Technology, Nanjing 210094, China}
\email{jl@njust.edu.cn}
\address{Department of Mathematics, Nanjing University,
Nanjing 210093, China}
\email{xcui@nju.edu.cn}
\address{Department of Mathematics, Nanjing University,
Nanjing 210093, China}
\email{penglu1991@gmail.com}
\date{\today}

\thanks{The first author is supported by the National Natural Science Foundation of China (Grants 11571166) and Start-up Foundation of Nanjing University of Science and Technology (No. AE89991/114). The third author is supported by the National Natural Science Foundation of China (Grants 11571166, 11631006, 11790272), the Project Funded by the Priority Academic Program Development of Jiangsu Higher Education Institutions (PAPD) and the Fundamental Research Funds for the Central Universities.}

\begin{abstract}
  In this paper, we show that if $(\mathbb{T}^{2},g)$ is a class A Lorentzian 2-torus with timelike poles, then there exists a Lipschitz foliation by complete future-directed timelike geodesics with any pre-assigned asymptotic direction in the interior of the stable time cone. This is done by constructing certain $C^{1,1}$ solutions to the equation $g(\nabla u,\nabla u)=-1$ on the Abelian cover $(\mathbb{R}^{2},g)$.
\end{abstract}
\maketitle

\section{Introduction}
This paper continues our study of dynamics of geodesics on class A Lorentzian 2-tori with poles initiated in \cite{Peng-Jin-Cui}. Instead of studying all timelike geodesics, we shall focus on certain geodesics whose lifts to the covering plane are complete timelike maximizers and possess a fixed ``timelike slope''. With the insight of Bangert, Schelling established the existence and order structure of such timelike geodesics in \cite{Sc} by generalizing Aubry-Mather theory to the case of class A Lorentzian 2-tori. However, in general, the timelike geodesics that we are interested in only constitute a topologically sparse set on the 2-torus.

Recently, Bangert \cite{Bangert10} showed the existence of foliations by lifted minimal geodesics with any prescribed homological direction for Riemannian 2-torus when poles arise. Due to similarity between the globally variational theories of Riemannian and class A Lorentzian 2-tori, we expect a corresponding phenomenon occurs for the latter case, which we now try to state.

First, we recall the notion of pole in the setting of Lorentzian geometry.
\begin{Def}
Let $(M,g)$ be a Lorentzian manifold, a point $p$ on $(M,g)$ is a timelike (resp. spacelike) pole if no timelike (resp. spacelike) geodesic $\gamma:[0,a)\rightarrow M$ starting from $p=\gamma(0)$ contains a pair of conjugate points.
\end{Def}

Let $(\mathbb{T}^{2},g)$ be a Lorentzian 2-torus and $\pi:\mathbb{R}^{2}\rightarrow\mathbb{T}^{2}$ its Abelian cover. We use the same symbol $g$ to denote the lifted Lorentz metric on $\mathbb{R}^{2}$. A Lorentzian 2-torus $(\mathbb{T}^{2},g)$ is class A if it is totally vicious and its Abelian cover $(\mathbb{R}^{2},g)$ is globally hyperbolic.

We say a future directed inextendible causal curve on $(\R^2,g)$, the Abelian cover of a class A Lorentzian 2-torus, has asymptotic direction if its distance from some straight line through the origin is bounded. It has been proved in \cite{Sc},\cite{Su-1} that the slopes of such straight lines forms an closed interval $[m^{-},m^+]$ and we denote its interior by $(m^{-},m^+)$.

With the terminologies introduced, our main result can be formulated into
\begin{TheoremX}\label{Th1}
Let $(\R^2,g)$ be the Abelian cover of a class A Lorentzian 2-torus having a timelike pole. For every asymptotic direction $\alpha\in(m^-,m^+)$, there exists a solution $u_{\alpha}\in C^{1}(\R^2,\R)$ to the equation
\begin{equation}\label{te}
g(\nabla u,\nabla u)=-1
\end{equation}
such that $\nabla u_{\alpha}$ is a $\Z^2$-invariant Lipschitz vector field and the integral curves of $\nabla u_{\alpha}$ are future-directed timelike geodesics having asymptotic direction $\alpha$.
\end{TheoremX}

There are a couple of things to note about Theorem \ref{Th1}:
\begin{Rem}\label{rm1}
It is easily seen, for instance Proposition \ref{prop1}, that every integral curve of $\nabla u_{\alpha}$ maximizes Lorentzian arclength between any two of its points. Such curves are called maximal and, by definition, have to be future-directed timeike geodesics.

Besides this, any integral curve of $\nabla u_{\alpha}$ is complete as a timelike geodesic since it has a timelike asymptotic direction $\alpha$, see Section 2.
\end{Rem}

\begin{Rem}\label{rm2}
If the asymptotic direction $\alpha$ is irrational, then $u_{\alpha}$ obtained by Theorem \ref{Th1} is unique up to a constant and the integral curves of $\nabla u_{\alpha}$ contains all future directed timelike maximal geodesics with asymptotic direction $\alpha$.

However, if $\alpha$ is rational, then $u_{\alpha}$ is unique up to a constant if and only if any of the future-directed timelike maximal geodesics with asymptotic direction $\alpha$ projects to closed geodesic under $\pi$. This phenomenon already occurs in Riemannian cases \cite{Bangert94}, \cite{Bangert10}, also compare our paper \cite{Jin-Cui}.
\end{Rem}

We say a future directed inextendible causal curve on a class A Lorentzian 2-torus $(\mathbb{T}^{2},g)$ has asymptotic direction if one of its lifts has. For any $u_{\alpha}$ obtained in Theorem \ref{Th1}, $\nabla u_{\alpha}$ is $\mathbb{Z}^{2}$-periodic, thus the integral curves of $\nabla u_{\alpha}$ project to geodesics on $\mathbb{T}^{2}$. This fact leads to the following immediate consequence of Theorem \ref{Th1}.

\begin{TheoremX}\label{Th2}
If $(\mathbb{T}^{2},g)$ has a timelike pole, then for every asymptotic direction $\alpha\in(m^-,m^+)$, there exists an oriented Lipschitz foliation of $\,\mathbb{T}^{2}$ by complete future-directed timelike geodesics with asymptotic direction $\alpha$. Moreover, any timelike geodesic constituting the foliation lifts to a timelike line on $(\mathbb{R}^{2},g)$.
\end{TheoremX}

The paper is organized as follows. In Section $2$, we introduce some facts about class A Lorentzian 2-torus that are useful in the proof of Theorem \ref{Th1}. In Section 3, we present necessary estimates and convergence proposition to construct desired global solution. Section 4 is devoted to the proof of our main theorem.

\section{Preliminaries}
In this section, we interpret elementary facts and previous results that are used later on. As in our previous paper \cite{Peng-Jin-Cui}, we always consider spacetimes (Lorentzian manifold with a time orientation) and future directed causal curves if there is no extra explanation. For more details, we recommend the readers to consult \cite{B-E-E},\cite{Sc},\cite{Su-1}-\cite{Su-4} and our former papers \cite{Jin-Cui},\cite{Peng-Jin-Cui}.

~\\
\subsection{Notations and definitions}
~\\
Let $\gamma:(a,b)\rightarrow M$ be a causal curve in a general spacetime $(M,g)$. The point $p$ is called a future (resp. past) endpoint of $\gamma$ if the limit $\lim_{t\rightarrow b^-}\gamma(t)$ (resp. $\lim_{t\rightarrow a^+}\gamma(t)$) exists and equals to $p$. A causal curve is called future (resp. past) inextendible if it has no future (resp. past) endpoint, and is called inextendible if it is both future and past inextendible. This leads to the following
\begin{Def}\label{Def_line}
A timelike (causal) geodesic $\gamma:[a,b)\rightarrow M$ is a timelike (causal) ray if it is future inextendible and maximal; a timelike (causal) geodesic $\gamma:(a,b)\rightarrow M$ is a timelike (causal) line if it is inextendible and maximal.
\end{Def}

Let $(\mathbb{T}^{2},g)$ be a class A Lorentzian 2-torus. Let $\pi:\mathbb{R}^{2}\rightarrow\mathbb{T}^{2}$ be the Abelian cover of the 2-torus, we shall denote the lift of $g$ on $\mathbb{R}^{2}$ by the same letter. The deck transformations $(\cong \Z^2)$ associated to $\pi$ act on $\R^2$ cocompactly. They are denoted by
\begin{equation}
T_{\mathbf{k}}:\R^2\rightarrow\R^2;\quad x\mapsto x+\mathbf{k}.
\end{equation}
Note that for every $\mathbf{k}\in\Z^2$, $T_\mathbf{k}$ is an isometry of $(\mathbb{R}^{2},g)$.

We denote $I^{\pm}(p),J^{\pm}(p)$ the usual chronological and causal sets, see also \cite{B-E-E} or \cite[Section 2.1]{Peng-Jin-Cui}; denote $L^g(\cdot)$ and $d(\cdot,\cdot)$ the length functional and distance function relative to $g$ respectively. Once and for all, we choose the standard Euclidean product metric on $\mathbb{T}^{2}:=\mathbb{T}\times\mathbb{T}$ or $\mathbb{R}^{2}:=\mathbb{R}\times\mathbb{R}$ as an auxiliary metric. Denote by $|\cdot|$ the norm of vectors relative to Euclidean metric and dist$_{|\cdot|}$ the distance induced by $|\cdot|$; denote by $L^{|\cdot|}(\cdot)$ the length functional relative to Euclidean metric. We do not clarify vectors on manifolds or tangent spaces since we use the Euclidean metric.

~\\
\subsection{Previous results on class A Lorentzian 2-tori}
~\\
To state results obtained in previous works, we recall the definitions of asymptotic direction. Let $S^{1}$ denote the unit circle on $(\mathbb{R}^{2},\text{dist}_{|\cdot|})$ and $\alpha\in S^{1}$. Set $\overline{\alpha}:=\{t\alpha\,\,|\,\,t\in[0,\infty)\}$.
\begin{Def}\label{def ap}
Let $\gamma:I\rightarrow \mathbb{R}^2$ be a future inextendible causal curve. If there exists constant $D<\infty$ and such that dist$_{|\cdot|}(\gamma(t)-\gamma(s),\overline{\alpha})\leq D$ for all $[s,t]\subseteq I$, we say $\gamma$ has an asymptotic direction $\alpha$. We call $\alpha\in S^1$ a rational asymptotic direction if $\overline{\alpha}\cap\mathbb{Z}^{2}$ contains a non-zero element. Otherwise, we call $\alpha$ irrational.
\end{Def}

In the following, $(\mathbb{R}^2,g)$ denotes the Abelian cover of a class A Lorentzian 2-torus. We orient the plane as usual, i.e., the counter-clockwise orientation is positive. Let $m^{\pm}$ be the asymptotic direction of the lifts of two families of future directed lightlike curves. The class A assumption implies that $m^-$ and $m^+$ are linearly independent, see \cite{Su-2}. We choose $\{m^-,m^+\}$ to be positively oriented. Now we set
$$
\mathfrak{T}=co(\overline{m}^-\cup \overline{m}^+),
$$
where $co(A)$ denotes the convex hull of $A\subseteq\mathbb{R}^{2}$. $\mathfrak{T}$ is called the stable time cone associated to $(\mathbb{T}^{2},g)$. From \cite[Proposition 8]{Su-1}, dist$_{|\cdot|}(J^+(x)-x,\mathfrak{T})$ is uniformly bounded over $x\in\mathbb{R}^{2}$, where $J^+(x)-x:=\{y-x\,\,|\,\,y\in J^+(x)\}$.

For any $\varepsilon>0$, we define
$$
\mathfrak{T}^\varepsilon:=\{h\in\mathfrak{T}\,\,|\,\,\text{dist}_{|\cdot|}(h,\partial\mathfrak{T})\geq\varepsilon|h|\},
$$
where $\partial\mathfrak{T}$ denotes the topological boundary of $\mathfrak{T}$. We call the closed arc $S^1\cap\mathfrak{T}$ the set of causal directions.  For $\alpha,\beta\in S^1\cap\mathfrak{T}$, we define $\alpha<\beta$ if $\{\alpha,\beta\}$ is positively oriented. From now on, we denote $(S^1\cap \mathfrak{T},<)$ by $[m^-,m^+]$ and let $(m^-,m^+):=[m^-,m^+]\backslash\{m^{\pm}\}$, which is called timelike directions.

With terminologies defined in Definition \ref{def ap}, for $\alpha\in (m^-,m^+)$, we use $\mathscr{R}_\alpha$ to denote the set of all future inextendible timelike rays on $(\mathbb{R}^2,g)$ with direction $\alpha$.

\begin{Rem}\label{uniform dis}
One can show that, see for instance \cite[Corollary 4.6]{Jin-Cui}, for timelike curves in $\mathscr{R}_\alpha$, the constant $D$ in Definition \ref{def ap} only depends on $g$.
\end{Rem}

The necessary facts on the structure of timelike lines with timelike directions is described in the following theorem. For more details, see \cite{Sc},\cite{Su-2} and \cite{Jin-Cui}.
\begin{The}[\cite{Sc}, \cite{Su-2}] \label{Sc-2}
Let $(\R^2,g)$ be the Abelian cover of a class A Lorentzian 2-torus, then:  \\
\begin{enumerate}
  \item There exist causal lines for every asymptotic direction $\alpha\in[m^-,m^+]$ and every causal line has an asymptotic direction $\alpha\in[m^-,m^+].$

  \item Let $\alpha\in(m^-,m^+)$ be an irrational asymptotic direction. Then any two distinct timelike lines with the same asymptotic direction $\alpha$ are disjoint.

  \item The asymptotic direction is continuous w.r.t. the $C^{0}$-topology on the space of causal lines, i.e. if $\gamma_{k}$ is a series of causal lines with asymptotic directions $\alpha_{k}$ and $\gamma_{k}$ converges to $\gamma$ in the $C^{0}$-topology with asymptotic direction $\alpha$, then $\alpha_{k}\rightarrow\alpha$ w.r.t. the topology defined on $[m^{-},m^{+}]$.
\end{enumerate}
\end{The}

The following proposition shows that on a class A 2-torus, a causal curve can not distort far from a segment.
\begin{Pro}[{\cite[Corollary 12]{Su-1}}]\label{lemma_Rdist}
Let $(\mathbb{R}^{2},g)$ be the Abelian cover of a class A Lorentzian 2-torus. There exists a constant $C(g)<\infty$ such that
\begin{equation*}
L^{|\cdot|}(\gamma)\leq B(g)\cdot|x-y|
\end{equation*}
for all $x,y\in\mathbb{R}^{2}$ and all causal curves $\gamma:I\rightarrow\mathbb{R}^{2}$ connecting $x$ and $y$.
\end{Pro}

Let Light$(M,g)$ denote the set of all future directed null vectors on $(\mathbb{R}^{2},g)$. For $\delta>0$, we define
$$
\text{Time}(\mathbb{R}^{2},g)^{\delta}:=\{v\text{ future directed}\,\,|\,\,\text{dist}_{|\cdot|}(v,\text{Light}(\mathbb{R}^{2},g))\geq\delta|v|\}
$$
and Time$(\mathbb{R}^{2},g)^{1,\delta}$ as the set of unit timelike vectors in Time$(\mathbb{R}^{2},g)^{\delta}$. The following theorem shows $\text{Time}(\mathbb{R}^2,g)^{1,\delta}$ is compact.

\begin{Pro}[{\cite[Theorem 3.3]{Jin-Cui}}]\label{compact_bundle}
Let $(\mathbb{R}^2,g)$ be the Abelian cover of a class A Lorentzian 2-torus, then there exists a constant $K(g)>0$ such that
\begin{equation*}
  |v|\leq \frac{K(g)}{\delta}
\end{equation*}
for any $v\in\text{Time}(\mathbb{T}^2,g)^{1,\delta}$.
\end{Pro}

The following proposition is used in the construction of solutions to Equation \eqref{te} whose gradient lines has a fixed direction.
\begin{Pro}[{\cite[Proposition 27]{Su-1}}]\label{future}
For every $C>0$ there exists a constant $0<K(C)<\infty$ such that
\begin{equation*}
  B^{|\cdot|}_C(y)\subseteq I^+(x)
\end{equation*}
for all $x,y\in\mathbb{R}^{2}$ with $y-x\in\mathfrak{T}$ and dist$_{|\cdot|}(y-x,\partial \mathfrak{T})\geq K(C)$, where $B^{|\cdot|}_C(y)$ is the closed ball of radius $C$ centered at $y$.
\end{Pro}

Now we list two important results showing the Lipschitz regularity of Lorentzian distance function in the case of class A Lorentzian 2-tori. They form the basis of our construction of global solutions to the timelike eikonal equation \eqref{te}.
\begin{The} [\cite{Su-4}, Proposition 4.1] \label{Su4:compact-result}
Let $(\mathbb{R}^{2},g)$ be the Abelian cover of a class A Lorentzian 2-torus. Then for all $\varepsilon>0$, there exist $\delta(\varepsilon)>0$ and $K(\varepsilon)<\infty$ such that
\begin{equation*}
\dot{\gamma}(t)\in\text{Time}(\mathbb{R}^{2},g)^{1,\delta(\varepsilon)}
\end{equation*}
for all unit speed maximizers $\gamma:[a,b]\rightarrow\mathbb{R}^{2}$ with $\gamma(b)-\gamma(a)\in\mathfrak{T}^{\varepsilon}\backslash B^{|\cdot|}_{K(\varepsilon)}(0)$ and all $t\in [a,b]$.
\end{The}

\begin{The}[\cite{Su-4}, Theorem 4.3] \label{Su4:d-Lip}
Let $(\mathbb{R}^{2},g)$ be the Abelian cover of a class A Lorentzian 2-torus. Then for any $\varepsilon>0$, there exist constants $K(\varepsilon),L(\varepsilon)<\infty$ such that
\begin{equation*}
(x,y)\mapsto d(x,y)
\end{equation*}
is $L(\varepsilon)$-Lipschitz (with respect to the Euclidean metric) on $\{(x,y)\in\mathbb{R}^{2}\times\mathbb{R}^{2}\,\,|\,\,y-x\in\mathfrak{T}^{\varepsilon}\backslash B^{|\cdot|}_{K(\varepsilon)}(0)\}$.
\end{The}

\begin{Rem}
The above two proposition were obtained in \cite{Su-4} in more general setting, but applies directly to our case, see \cite[Proposition 2.5]{Su-4}.
\end{Rem}

\vspace{0.4cm}
\subsection{Lorentzian distance functions based at a pole}
~\\
First, we observe that timelike poles on $(\mathbb{R}^{2},g)$ are exactly lifts of timelike poles on $(\mathbb{T}^{2},g)$. Let $p\in\mathbb{R}^{2}$ be a timelike pole on $(\mathbb{R}^{2},g)$. We define $d_{p}:I^{-}(p)\cup I^{+}(p)\rightarrow\mathbb{R}$ by
\begin{equation*}
d_{p}(x)=\left\{
           \begin{array}{ll}
             d(x,p), & \hbox{$x\in I^{-}(p)$;} \\
             d(p,x), & \hbox{$x\in I^{+}(p)$.}
           \end{array}
         \right.
\end{equation*}

The definition of class A Lorentzian 2-tori implies that $(\mathbb{R}^{2},g)$ is globally hyperbolic. Then from our former paper \cite{Peng-Jin-Cui}, $d_{p}$ is $C^{1}$ on its domains. We formulate this fact into
\begin{Pro}\label{smooth}
Let $(\mathbb{R}^{2},g)$ be a globally hyperbolic plane and $p\in\mathbb{R}^{2}$ a timelike pole. Then the function $d_{p}$ is a $C^{1}$ solution to the timelike eikonal equation \eqref{te} on $I^{-}(p)\cup I^{+}(p)$.
\end{Pro}

\begin{proof}
By symmetry, we shall only prove the conclusion on $I^{+}(p)$. From \cite[Corollary 5.3]{Peng-Jin-Cui}, $exp_{p}$ is a diffeomorphism from some open set $0\notin\mathcal{F}^{+}(p)\subset T_{p}\mathbb{R}^{2}$ onto $I^+(p)$. Hence $d_{p}(x)=|exp_{p}^{-1}(x)|_{g}$ is $C^{1}$ on $I^{+}(p)$.

Since $(\mathbb{R}^{2},g)$ is globally hyperbolic, through any $x\in I^+(p)$, there is a unique future-directed maximal timelike geodesic segment $\gamma_{x}:[0,d_{p}(x)]\rightarrow\mathbb{R}^{2}$ starting from $p$. It follows that $d_{p}(\gamma_{x}(t))=t$ and $|\nabla d_{p}(x)|_{g}=|\dot{\gamma}_{x}(d_{p}(x))|_{g}=1$.
\end{proof}

\begin{Rem}
If $p$ is a timelike pole on a globally hyperbolic plane $(\mathbb{R}^{2},g)$, then any timelike geodesic starting from $p$ is a timelike ray.
\end{Rem}

Moreover, the appearance of timelike poles leads to the existence of periodic timelike geodesics. Here we call a complete timelike geodesic $\gamma$ on $(\mathbb{R}^2,g)$ periodic if there exist $a>0$ and $\mathbf{k}\in\mathbb{Z}^2\backslash\{0\}$ such that $\gamma(t)+\mathbf{k}=\gamma(t+a)$ for all $t\in\R$. The above fact is the main subject of our first paper \cite{Peng-Jin-Cui} and can be formulated into

\begin{The}[\cite{Peng-Jin-Cui}, Theorem A]\label{exist per}
Let $(\R^{2},g)$ be the Abelian cover of a class A Lorentzian 2-torus and $p$ any timelike pole on it. Then for any rational asymptotic direction $\alpha\in(m^-,m^+)$, there exists a unique periodic timelike line passing through $p$ with asymptotic direction $\alpha$.
\end{The}

\vspace{0.4cm}
\subsection{Timelike eikonal equations on spacetimes}
~\\
Let $(M,g)$ be a general spacetime, the the canonical Hamilton-Jacobi type equation \eqref{te} associated with $(M,g)$ is called timelike eikonal equation. In this paper, we discuss solutions to timelike equation \eqref{te} defined on open subsets of $M$.

Let $u\in C^1(U,\R)$ be a solution to \eqref{te}. Suppose that the vector field $\nabla u$ is future directed, then by Peano's existence theorem from ODE, for any $x\in U$, there exists at least one future directed timelike integral curve $\gamma_{x}:(a_x,b_x)\rightarrow U$ of $\nabla u$ with $\gamma_{x}(0)=x$, where $(a_x,b_x)$ denotes the maximal interval of existence. The following proposition shows that $\gamma_{x}$ is maximal in the sense that defined in Remark \ref{rm1}.

\begin{Pro}\label{prop1}
$\gamma_{x}$ is maximal in the sense that for any $[s,t]\subseteq (a_x,b_x)$ and any future directed piecewise $C^{1}$ timelike curve $\xi:[a,b]\rightarrow U$ such that $u(\xi(a))-u(\xi(b))=u(\gamma_{x}(s))-u(\gamma_{x}(t))$,
$$
L^g(\gamma_{x}|_{[s,t]})\geq L^g(\xi).
$$
\end{Pro}

\begin{proof}
Since $\dot{\gamma}_{x}$ is an integral curve of $\nabla u$, we have
\begin{equation} \label{E4}
\begin{split}
|\dot{\gamma}_x|_g&=\sqrt{-g(\dot{\gamma}_x,\dot{\gamma}_x)}=\sqrt{-g(\nabla u,\nabla u)}=1,\\
L^g(\gamma_x|_{[s,t]})&=\int_s^t |\dot{\gamma}_x|_g\,\mathrm{d}\tau=t-s=u(\gamma_x(s))-u(\gamma_x(t)).
\end{split}
\end{equation}
We assume $\xi:[a,b]\rightarrow U$ to be a $C^1$ future directed timelike curve such that $u(\xi(a))-u(\xi(b))=u(\gamma_{x}(s))-u(\gamma_{x}(t))$. Since both $\nabla u$ and $\dot{\xi}$ are future-directed timelike, we have that $g(\nabla u, \dot{\xi})<0$ and
\begin{equation} \label{E5}
-\mathrm{d}u(\dot{\xi})=-g(\nabla u, \dot{\xi})\geq |\nabla u|_g|\dot{\xi}|_g=|\dot{\xi}|_g,
\end{equation}
here the inequality follows from the reverse Cauchy-Schwarz inequality for causal vectors (\cite{S-W},2.2.1). Combining \eqref{E4} and \eqref{E5}, we get
\begin{equation}
\begin{split}
L^g(\gamma_x|_{[s,t]}) & =u(\gamma_x(s))-u(\gamma_x(t))=u(\xi(a))-u(\xi(b)) \nonumber \\
&=\int_a^b -\mathrm{d}u(\dot{\xi})\,\mathrm{d}\tau\geq\int_a^b |\dot{\xi}|_g\,\mathrm{d}\tau=L^g(\xi),
\end{split}
\end{equation}
which implies the maximality of $\gamma_x$.
\end{proof}

\begin{Rem}
Similar results are well-known in the setting of weak KAM theory and calibrated geometry. In terms of calibrated geometry, $du$ is called a calibration and $\gamma_x$ is calibrated by $du$.
\end{Rem}

\section{Construction of global solutions}
As in \cite{Peng-Jin-Cui}, we use $\alpha$ to denote asymptotic directions, use $m^{\pm}$ to denote the asymptotic directions of two families of lightlike curves. Rational (resp. irrational) directions are defined similar as in \cite{Peng-Jin-Cui}. To begin our construction, we reformulate the main result of \cite{Peng-Jin-Cui} into the following

\begin{Pro}
Let $(\R^{2},g)$ be the Abelian cover of a class A 2-torus and $p\in\R^{2}$ a timelike pole, then for any rational asymptotic direction $\alpha\in(m^-,m^+)$, there exists a unique future directed periodic timelike line $\gamma_{p,\alpha}$ passing through $p$ with asymptotic direction $\alpha$.
\end{Pro}

For any two rational timelike directions $\alpha^-<\alpha^+$, we denote by $D^\pm_{\alpha^-,\alpha^+}(p)\subseteq I^\pm(p)$ the open angular domain bounded by $\gamma_{p,\alpha^-}$ and $\gamma_{p,\alpha^+}$. We note that
$$
p\in\overline{D^\pm_{\alpha^-,\alpha^+}(p)}\subseteq J^\pm(p).
$$
Let $O^+(p,1):=\{x\in \R^2: d(p,x)>1\}$, we set
\begin{equation*}
F_{\alpha^-,\alpha^+}(p):=D^+_{\alpha^-,\alpha^+}(p)\bigcap O^+(p,1).
\end{equation*}
Using that the outer ball $O^{+}(p,1)$ is open in $I^+(p)$, $F_{\alpha^-,\alpha^+}(p)$ is also an open subset of $I^+(p)$. It follows from the notions above that
\begin{Pro}\label{lemma_complete_ray}
Let $(\mathbb{R}^{2},g)$ be the Abelian cover of a class A Lorentzian 2-torus and $p\in\mathbb{R}^{2}$ a timelike pole, then
\begin{enumerate}
  \item through any $x\in D^+_{\alpha^-,\alpha^+}(p)$, there exists a unique timelike ray $\gamma_{x}:[0,\infty)\rightarrow\{p\}\cup D^+_{\alpha^-,\alpha^+}(p)$, parameterized by $g$-arc length, starting from $p$.
  \item if $p^{\prime}\in \overline{D^-_{\alpha^-,\alpha^+}(p)}$ is a timelike pole,
      \begin{equation*}
      D^+_{\alpha^-,\alpha^+}(p)\subseteq D^+_{\alpha^-,\alpha^+}(p^{\prime}),\quad F_{\alpha^-,\alpha^+}(p)\subseteq F_{\alpha^-,\alpha^+}(p^{\prime}).
      \end{equation*}
\end{enumerate}
\end{Pro}

\begin{proof}
(1) For any $x\in D^+_{\alpha^-,\alpha^+}(p)$, by global hyperbolicity of $(\mathbb{R}^{2},g)$, there exists a future-directed maximal timelike geodesic segment $\hat{\gamma}_{x}$ connecting $p$ with $x$. Since $p$ is a timelike pole, then from Proposition \ref{smooth} or \cite[Corollary 5.3 and Remark 5.4]{Peng-Jin-Cui}, we know that $\hat{\gamma}_{px}$ is unique and can be extended to a future-directed timelike ray $\gamma_{x}$.

By the Lorentzian version of Morse's crossing lemma \cite[Section 4]{Jin-Cui}, two maximal timelike geodesic segments will not intersect twice except that the intersections occurred at two endpoints. So $Im(\gamma_{x})\subseteq \{p\}\cup D^+_{\alpha^-,\alpha^+}(p)$. In this case, since $\gamma_{x}$ is a timelike ray, by \cite[Lemma 4.4, Theorem 4.5]{Jin-Cui}, $\gamma_{x}$ has an asymptotic direction $\alpha\in[\alpha^{-},\alpha^{+}]$. Thus we apply Theorem \ref{Su4:compact-result} and Proposition \ref{compact_bundle} to conclude that the domain of the affine parameter of $\gamma_{x}$ is $[0,\infty)$.

(2) Since $p^{\prime}$ is a timelike pole, there exist unique periodic timelike lines $\gamma_{p^{\prime},\alpha^\pm}$ passing through $p^{\prime}$ with asymptotic direction $\alpha^\pm$. Thus, $\gamma_{p^{\prime},\alpha^{\pm}}$ and $\gamma_{p,\alpha^{\pm}}$ either coincide or have no intersection since they have the same asymptotic direction. Together with $D^+_{\alpha^-,\alpha^+}(p)\subseteq I^+(p)\subset I^+(p^{\prime})$, we conclude that $D^+_{\alpha^-,\alpha^+}(p)\subseteq D^+_{\alpha^-,\alpha^+}(p^{\prime})$.

Since $p^{\prime}\in \overline{D^-_{\alpha^-,\alpha^+}(p)}\subseteq J^-(p)$, one deduces $O^+(p,1)\subseteq O^+(p^{\prime},1)$ by the reverse triangular inequality. Therefore, $F_{\alpha^-,\alpha^+}(p)\subseteq F_{\alpha^-,\alpha^+}(p^{\prime})$.
\end{proof}

\begin{Rem}
From the proof, it is easy to see that the time dual version of $(1)$ is also true, i.e. through any $x\in D^-_{\alpha^-,\alpha^+}(p)$, there exists a unique timelike ray $\gamma_{x}:(-\infty,0]\rightarrow D^-_{\alpha^-,\alpha^+}(p)\cup\{p\}$ with $\gamma_{x}(0)=p$.
\end{Rem}

The next proposition shows the equi-Lipschitz property of $\{\nabla d_{p_{i}}\}_{i\in\mathbb{N}}$ when $p_{i},i\in\mathbb{N}$ are timelike poles and $p_{i+1}\in \overline{D^-_{\alpha^-,\alpha^+}(p_{i})}$.
\begin{Pro} \label{prop2}
Let $(\mathbb{R}^{2},g)$ be the Abelian cover of a class A Lorentzian 2-torus and $p\in\mathbb{R}^{2}$ a timelike pole as above. If $\alpha^{\pm}\in\mathfrak{T}^{\varepsilon}$, there exists $H(\varepsilon)>0$ such that for any timelike pole $p^{\prime}\in \overline{D^-_{\alpha^-,\alpha^+}(p)}$, $\nabla d_{p^{\prime}}$ is $H$-Lipschitz with respect to the Euclidean metric on $F_{\alpha^-,\alpha^+}(p)$.
\end{Pro}

\begin{proof}
We shall focus on the case when $p^{\prime}=p$, general cases are completely similar. To deduce $\nabla d_p$ is $L_{0}$-Lipschitz, it is sufficient to show that
\begin{equation}\label{hess es}
-H\cdot Id\leq\text{Hess }[d_{p}](x)\leq H\cdot Id
\end{equation}
in the sense of upper and lower functions.

For any $x\in F_{\alpha^-,\alpha^+}(p)$, we set $d_{p}(x)=\tau>1$. By Lemma \ref{lemma_complete_ray}, there exists a unique timelike ray $\gamma_{px}:[0,\infty)\rightarrow\{p\}\cup D^+_{\alpha^-,\alpha^+}(p)$ such that $\gamma_{px}(0)=p$ and $\gamma_{px}(\tau)=x$. In a neighborhood $U_x$ of $x$, we define
\begin{equation}\label{def spt}
\begin{split}
f^-_{px}(y):=\tau-1+d(\gamma_{px}(\tau-1),y),\\
f^+_{px}(y):=\tau+1-d(y,\gamma_{px}(\tau+1)).
\end{split}
\end{equation}
If $U_x\subset I^{+}(\gamma_{x}(\tau-1))\cap I^{-}(\gamma_{x}(\tau+1))$, by the reverse triangular inequality, for any $y\in U_x$,
$$
f^-_{px}(y)\leq d_p(y)\leq f^+_{px}(y)
$$
with both equalities at $y=x$. Thus, $f^-_{px}$ (resp. $f^+_{px}$) is a lower (resp. upper) support function for $d_p$ at $x$. Since $\gamma_{px}$ maximizes the distance between its points, $\gamma_{px}|_{[\tau-1,\tau]}$ and $\gamma_{px}|_{[\tau,\tau+1]}$ are free of cut points. This implies $f^{\pm}_{px}$ are smooth near $x$, see \cite[Proposition 9.29]{B-E-E}.

Since we are in the case of dimension two, by Lorentzian version of the Hessian Comparison Theorem \cite[Section 9.4]{A-M-R}, we can give a lower bound for the Hessian (defined in terms of the Levi-Civita connection w.r.t. $g$) of
\begin{equation*}
y\mapsto d(\gamma_{px}(\tau-1),y)
\end{equation*}
in terms of the bounds for $|\dot{\gamma}_{px}|$ and Gauss curvature along $\gamma_{px}|_{[\tau-1,\tau]}$. Since Gauss curvature of $g$ is uniformly bounded and $\alpha^{\pm}\in\mathfrak{T}^{\varepsilon}$, we use Theorem \ref{Su4:compact-result} to conclude that there is $H>0$ such that
$$
-H\cdot Id\leq\text{Hess }[y\mapsto d(\gamma_{px}(\tau-1),y)](x)=\text{Hess }[f^{-}_{px}](x),
$$
where the second equality follows from equation \eqref{def spt}. This shows the left hand of inequalities \eqref{hess es}. To finish the proof, we observe that the dual argument holds for $f^{+}_{px}$.
\end{proof}

\begin{Rem}
We also refer to \cite[Theorem 3.1]{E-G-K} for the Lorentzian Hessian comparisons. Same method can be used to prove the locally semiconvex property for certain generalized Lorentzian distance functions, see \cite[Proposition 3.1]{A-G-H1} and \cite[Theorem 5.10]{Jin-Cui}.
\end{Rem}

\textit{Assumptions:}
We say a sequence of timelike poles $\{p_i\}_{i\in\mathbb{N}}$ on $(\mathbb{R}^2,g)$ satisfies condition $(*)$ if there exist $\alpha^\pm\in(m^-,m^+)$ with $\alpha^-<\alpha^+$ such that
\begin{equation*}
p_{i+1}\in D^-_{\alpha^-,\alpha^+}(p_{i}),\hspace{0.2cm} d(p_{i+1},p_i)\geq1
\end{equation*}
for all $i\in\mathbb{N}$ and
\begin{equation*}
\bigcup\limits_{i\in\mathbb{N}} D^+_{\alpha^-,\alpha^+}(p_i)=\mathbb{R}^2.
\end{equation*}

\begin{Rem}\label{rm3}
If $\alpha^\pm\in\mathfrak{T}^{\varepsilon_{0}}$, then for any $\varepsilon<\varepsilon_{0}$ and any compact set $K\subseteq \mathbb{R}^2$, there exists $N\in\mathbb{N}$ such that $x-p_i\in \mathfrak{T}^{\varepsilon}$ for all $x\in K, i\geq N$.
\end{Rem}

The following proposition shows that if $\{p_{i}\}_{i\in\mathbb{N}}$ satisfies $(*)$, then up to constants, $d_{p_{i}}$ converges to a $C^{1,1}$ global solution to the timelike eikonal equation \eqref{te}.

\begin{Pro} \label{prop3}
Let $\alpha^\pm\in(m^-,m^+)$ with $\alpha^-<\alpha^+$ and $\{p_i\}_{i\in\N}$ be a sequence of timelike poles in $(\R^2,g)$ satisfying condition $(*)$. Set $u_{i}:=-d_{p_i}, i\in\N$. If for some $x_{0}\in \R^2$, the sequence $\{u_{i}(x_{0})\}_{i\in\N}$ has a convergent subsequence, then there exists a subsequence $\{u_{i_{k}}\}_{k\in\N}$ of $\{u_{i}\}_{i\in\N}$ converges to a solution $u\in C^1(\R^2,\R)$ to Equation \eqref{te} in the $C^{1}$ topology, i.e., $\lim_{k\rightarrow\infty} u_{i_{k}}=u$ and $\lim_{k\rightarrow\infty}\nabla u_{i_{k}}=\nabla u$ uniformly on compact subsets.
\end{Pro}

\begin{proof}
Since $p_i$ is a timelike pole for any $i\in\N$, $u_i=-d_{p_i}$ is a smooth solution to Equation \eqref{te} in $I^+(p_i)$ by Proposition \ref{smooth}. We assume $\alpha^{\pm}\in\mathfrak{T}^{\varepsilon_{0}}$ and set $\varepsilon=\frac{\varepsilon_{0}}{2}$.

Now we fix any compact set $K\subseteq \R^2$. By Remark \ref{rm3}, there exist $R(\varepsilon)>0$ and $N(\varepsilon,K)\in\mathbb{N}$ such that if $x\in K,i\geq N$, then $x-p_i\in \mathfrak{T}_{\varepsilon}\backslash B_{R(\varepsilon)}(0)$. Hence, by Proposition \ref{future}, we may assume $K\subseteq I^{+}(p_{i})$ for $i\geq N$ by taking a larger $N$. Let $\gamma_{p_ix}$ be the unit speed geodesic ray starting from $p_i$ and going through $x$, then by Theorem \ref{Su4:compact-result}, there exists $\delta(\varepsilon)>0$ such that
\begin{equation}\label{eqn3-2}
\nabla u_i(x)=\dot{\gamma}_{p_ix}(0)\in\text{Time}(\R^2,g)^{1,\delta}.
\end{equation}
Hence from \eqref{eqn3-2} and Proposition \ref{compact_bundle}, $\{|\nabla u_i|\}_{i\geq N}$ is uniformly bounded on $K$. Using Proposition \ref{prop2}, $\{\nabla u_i\}_{i\in\N}$ is also uniformly Lipschitz on $K$.

Thus by Ascoli-Arzela theorem and a diagonal sequence argument, there exists a subsequence $\{\nabla u_{i_k}\}_{k\in\N}$ of $\{\nabla u_i\}_{i\in\N}$ converges uniformly on compact sets to a locally Lipschitz vector field $X$ in $\R^2$, that is, $\lim_{k\rightarrow\infty}\nabla u_{i_{k}}=X$. Since $\{u_{i}(x_{0})\}_{i\in\N}$ has a convergent subsequence, we can assume that $\lim_{k\rightarrow\infty}u_{i_k}(x_{0})$ exists.

For any (piecewise) $C^1$ curve $\gamma:[0,1]\rightarrow \R^2$ with $q=\gamma(0)$, we have
\begin{equation*}
\lim_{k\rightarrow\infty}[u_{i_k}(\gamma(1))-u_{i_k}(x_{0})]=\lim_{k\rightarrow\infty}\int_0^1 g(\nabla u_{i_k}|_{\gamma(t)},\dot{\gamma}(t))dt=\int_0^1 g(X|_{\gamma(t)},\dot{\gamma}(t))dt.
\end{equation*}
Note that $u(\gamma(1))-u(x_{0})=\int_0^1 g(\nabla u|_{\gamma(t)},\dot{\gamma}(t))dt$. Since $\lim_{k\rightarrow\infty}u_{i_k}(x_{0})$ exists, the arbitrary choice of $\gamma$ implies that $\{u_{i_{k}}\}_{k\in\N}$ converges uniformly on compact sets to a function $u: \R^2\rightarrow\R$ and $\text{grad}u=X$. This completes the proof.
\end{proof}

\begin{Rem}\label{uniform Lip}
The global solution $u$ to Equation \eqref{te} constructed above, are Lipschitz with constant $L(\varepsilon)$ since for any compact set $K$, $u_i$ are Lipschitz with constant $L(\varepsilon)$ on $K\subseteq\{x\in\R^2|x-p_i\in\mathfrak{T}^{\varepsilon}\backslash B_{R(\varepsilon)}(0)\}$ when $i$ is large enough (see Theorem \ref{Su4:d-Lip}).
\end{Rem}

The following well-known limit curve lemma is useful in our proof.
\begin{Lem}[{\cite[Lemma 2.1]{Galloway}}]\label{lemma_limit_curve}
Let $\gamma_i:\mathbb{R}\rightarrow M$ be a sequence of inextendible causal curves parameterized with Euclidean arc-length. If $p\in M$ is an accumulation point of the sequence $\{\gamma_i(0)\}$, then there exists an inextendible causal curve $\gamma:\mathbb{R}\rightarrow M$ such that $\gamma(0)=p$ and a subsequence $\{\gamma_j\}$ which converges to $\gamma$ uniformly (w.r.t. $g_R$) on compact intervals of $\mathbb{R}$. $\gamma$ is called a limit curve of $\gamma_i$.
\end{Lem}

We need the following elementary lemmas to proceed.

\begin{Lem} \label{lemma_e_cone}
For any $\alpha\in(m^-,m^+), C>0$ and any timelike ray $\gamma\in\mathscr{R}_\alpha$, where $\gamma:[0,\infty)\rightarrow (\mathbb{R}^2,g)$ is parameterized by Euclidean arc-length, there exist positive constants $\varepsilon=\varepsilon(\alpha,g), T=T(C,\alpha,g)$ such that for all $t\geq T,x\in B^{|\cdot|}_C(\gamma(t))$,
\begin{equation*}
x-\gamma(0)\in \mathfrak{T}^\varepsilon.
\end{equation*}
\end{Lem}

\begin{proof}
Since $\gamma:[0,\infty)\rightarrow \mathbb{R}^2$ is parameterized by $g_R$-arc length, then applying Proposition \ref{lemma_Rdist}, for any $t>0$,
\begin{equation}\label{eq_dist1}
|\gamma(0)-\gamma(t)|\geq\frac{t}{B(g)}.
\end{equation}
By Remark \ref{uniform dis}, there exists a number $D(\alpha,g)<\infty$ such that for any $t>0$,
\begin{equation}\label{eq_dist2}
\text{dist}_{|\cdot|}(\gamma(t)-\gamma(0),\overline{\alpha})\leq D(g).
\end{equation}
Hence there is $h\in\overline{\alpha}$ such that $|h|\geq|\gamma(0)-\gamma(t)|-D(g)$. Since $\alpha\in (m^-,m^+)$,
\begin{equation}\label{eq_dist3}
\theta_\alpha:=\text{dist}_{|\cdot|}(\alpha,\overline{m}^{\pm})>0.
\end{equation}
By \eqref{eq_dist1}-\eqref{eq_dist3}, for fixed $C>0$ set
$$
T:=\frac{2B(g)}{\theta_\alpha}[(\theta_\alpha+1)D(g)+C].
$$
If $t\geq T$ and $x\in B^{|\cdot|}_C(\gamma(t))$, then dist$_{|\cdot|}(x-\gamma(0),\overline{m}^{\pm})\geq\frac{\theta_\alpha}{2}|\gamma(t)-\gamma(0)|$. We choose $\varepsilon=\frac{\theta_\alpha}{2}$ to complete the proof.
\end{proof}

\begin{Lem}\label{lemma_p_bounded}
Let $\alpha\in(m^-,m^+)$ and $\gamma:[0,\infty)\rightarrow \mathbb{R}^2$ be a timelike ray in $\mathscr{R}_\alpha$. Let $\{t_i\}_{i\in\mathbb{Z}_+}$ be a sequence with $\lim_{i\rightarrow \infty}t_i=\infty$ and $\{\mathbf{k}_i\}_{i\in\mathbb{Z}_+}\subseteq\mathbb{Z}^{2}$ be a sequence of integer vectors such that there exists a constant $C>0$ satisfying $T_{\mathbf{k}_i}\circ\gamma(t_i)\in B^{|\cdot|}_C(\gamma(0))$ for all $i\in\mathbb{Z}_+$. Then there exists a constant $Q=Q(C,\alpha,g)$ such that
\begin{equation*}
\text{dist}_{|\cdot|}(T_{\mathbf{k}_i}\circ\gamma(0)-T_{\mathbf{k}_j}\circ\gamma(0),\overline{\alpha})\leq Q
\end{equation*}
for all $1\leq i<j$.
\end{Lem}

\begin{proof}
Using Remark \ref{uniform dis}, there exists a constant $D(g)<\infty$ such that for all $s>t\geq 0$,
\begin{equation}\label{eq_bound1}
\text{dist}_{|\cdot|}(\gamma(s)-\gamma(t),\overline{\alpha})\leq D(g).
\end{equation}
For any $1\leq i<j$, by assumption we have
\begin{equation}\label{eq_bound2}
\begin{split}
|T_{\mathbf{k}_i}\circ\gamma(0)-T_{\mathbf{k}_i}\circ T_{\mathbf{k}_j}\circ\gamma(t_{j})|\leq C,\\
|T_{\mathbf{k}_j}\circ\gamma(0)-T_{\mathbf{k}_j}\circ T_{\mathbf{k}_i}\circ\gamma(t_{i})|\leq C.
\end{split}
\end{equation}
Thus, by using \eqref{eq_bound1}, \eqref{eq_bound2} and the commutativity of translations $\Gamma$, we have
\begin{eqnarray*}
  &&\text{dist}_{|\cdot|}(T_{\mathbf{k}_i}\circ\gamma(0)-T_{\mathbf{k}_j}\circ\gamma(0),\overline{\alpha}) \\
  &\leq & \text{dist}_{|\cdot|}(T_{\mathbf{k}_j}\circ T_{\mathbf{k}_i}\circ\gamma(t_{i})-T_{\mathbf{k}_i}\circ T_{\mathbf{k}_j}\circ\gamma(t_{j}),\overline{\alpha})+2C \\
  &=&\text{dist}_{|\cdot|}(\gamma(t_i)-\gamma(t_j),\overline{\alpha})+2C\\
  &\leq& Q:= D(g)+2C. \\
\end{eqnarray*}
\end{proof}

Now we formulate the main result of this section.
\begin{Pro} \label{prop4}
Let $p$ be a timelike pole on $(\mathbb{R}^2,g)$. Then for every unit speed timelike ray $\gamma\in\mathscr{R}_{\alpha}$ starting from $p$ and every sequence $\{t_i\}_{i\in\mathbb{Z}_+}$ with $\lim_{i\rightarrow\infty}t_i=\infty$, there exists a solution $u\in C^{1,1}(\mathbb{R}^2,\mathbb{R})$ to the Equation \eqref{te} with the following property: There exists a sequence $\{\mathbf{k}_i\}_{i\in\mathbb{Z}^{+}}\in\Z^{2}$ such that the sequence of geodesics $t\rightarrow T_{\mathbf{k}_i}\circ \gamma(t+t_i)$ converges to an integral curve of $\nabla u$.
\end{Pro}

\begin{proof}
Since $\Gamma$ acts cocompactly on $\mathbb{R}^2$, there exists a constant $0<C<\infty$ satisfying that for each $i\in\mathbb{Z}_+$ there exists $\mathbf{k}_i\in\Z^{2}$ such that
\begin{equation}\label{eq_prop4_1}
\text{dist}_{|\cdot|}(p,T_{\mathbf{k}_i}\circ\gamma(t_i))\leq C.
\end{equation}
Since $\gamma\in\mathscr{R}_\alpha$ and $\lim_{i\rightarrow\infty}t_i=\infty$, then for any $T>0$ there exist $N=N(T)>0$ such that for all $i>N$,
\begin{equation}\label{eq_prop4_2}
L^{|\cdot|}(\gamma|_{[0,t_i]})>T.
\end{equation}

Thus by \eqref{eq_prop4_1}, \eqref{eq_prop4_2} and Lemma \ref{lemma_e_cone}, there exist positive constants $\varepsilon=\varepsilon(\alpha,g)$, $T=T(C,\alpha,g)$, $N=N(T)$ such that for all $i>N$ and $x\in B^{|\cdot|}_{2C}(\gamma(t_i))$,
\begin{equation*}
x-\gamma(0)\in \mathfrak{T}^\varepsilon.
\end{equation*}
Note that $T_{\mathbf{k}_i}\in\Gamma$ is also a Euclidean isometry and $B^{|\cdot|}_C(p)\subseteq B^{|\cdot|}_{2C}(T_{\mathbf{k}_i}\circ\gamma(t_i))$, hence for all $i>N$ and $x\in B^{|\cdot|}_C(p)$, we have
\begin{equation}\label{eq_prop4_3}
x-T_{\mathbf{k}_i}(p)\in \mathfrak{T}^\varepsilon.
\end{equation}
Therefore, by \eqref{eq_prop4_3} and Theorem \ref{Su4:d-Lip}, there are $L(\varepsilon)<\infty$ and $N'$ (larger than $N$) such that $d(T_{\mathbf{k}_i}(p),x)$ is $L(\varepsilon)$-Lipschitz on $B^{|\cdot|}_C(p))$ for all $i\geq N'$.
Thus, for all $i\geq N'$, we have
\begin{equation}\label{eq_prop4_4}
\begin{split}
|d(T_{\mathbf{k}_i}(p),p)-t_i|=|d(T_{\mathbf{k}_i}(p),p)-d(T_{\mathbf{k}_i}(p),T_{\mathbf{k}_i}\circ\gamma(t_i))|\\
\leq L(\varepsilon)\cdot|p-T_{\mathbf{k}_i}\circ\gamma(t_i)|\leq L(\varepsilon)\cdot C.
\end{split}
\end{equation}

For each $i\in\mathbb{Z}_+$, denote by $\gamma_i:[-t_i,\infty)\rightarrow \mathbb{R}^2$, $\gamma_i(t)=T_{\mathbf{k}_i}\circ \gamma(t_i+t)$, then for any $t\geq -t_i, d(\gamma_i(-t_i),\gamma_i(t))=d(T_{\mathbf{k}_i}\circ\gamma(0),T_{\mathbf{k}_i}\circ \gamma(t_i+t))=t_i+t$. Let $u_i: I^+(T_{\mathbf{k}_i}(p))\rightarrow\mathbb{R}$ be defined by
\begin{displaymath}
u_i(x):=t_i-d(T_{\mathbf{k}_i}(p),x),
\end{displaymath}
then $u_i$ is a smooth solution to the timelike eikonal equation \eqref{te} and $u_i(\gamma_i(t))=t_i-d(T_{\mathbf{k}_i}(p),\gamma_i(t))=-t$ for all $t\geq -t_i$.
Since \eqref{eq_prop4_4} implies $|u_i(p)|=|t_i -d(T_{\mathbf{k}_i}(p),p)|\leq C\cdot L(\varepsilon)$ when $i\geq N'$, the sequence $\{u_i(p)\}_{i\in\mathbb{Z}_+}$ has a convergent subsequence.

From Lemma \ref{lemma_p_bounded}, all $T_{\mathbf{k}_i}(p)$ has a bounded distance from the line $\overline{\alpha}$. Note that for any timelike line $\zeta$ with asymptotic direction $\alpha$, we have $I^\pm(\zeta)=\mathbb{R}^2$. We choose $\alpha^\pm\in (m^-,m^+)$ such that $\alpha^-<\alpha<\alpha^+$. Then, by taking a subsequence, Lemma \ref{lemma_e_cone} implies that $\{T_{\mathbf{k}_i}(p)\}$ satisfies the condition $(*)$ in Proposition \ref{prop3}.
Thus by Proposition \ref{prop3}, there exists a subsequence $\{u_{i_k}\}$ of $\{u_i\}$ converges to a solution $u\in C^{1,1}(\mathbb{R}^2,\mathbb{R})$ to the Lorentzian eikonal equation in the $C^1$ topology.

Finally, since dist$_{|\cdot|}(p,\gamma_i(0))\leq C$, by taking a subsequence, we could also assume that $\gamma_{i_k}|_{[-t_{i_k},\infty)}$ converges to a geodesic $\tilde{\gamma}:\mathbb{R}\rightarrow M$. (use Limit Curve Lemma \ref{lemma_limit_curve}) For any $[s,t]\subseteq \mathbb{R}$, we have
\begin{equation}
u(\tilde{\gamma}(t))-u(\tilde{\gamma}(s))=\lim_{k\rightarrow \infty} [u_{i_k}(\gamma_{i_k}(t))-u_{i_k}(\gamma_{i_k}(s))]=s-t.
\end{equation}
Thus, we obtain that $(u\circ \tilde{\gamma})'(t)=-1=g(\nabla u|_{\tilde{\gamma}(t)},\dot{\tilde{\gamma}}(t))$. By $g(\dot{\tilde{\gamma}}(t),\dot{\tilde{\gamma}}(t))=-1$, we have $\dot{\tilde{\gamma}}(t)=\nabla u|_{\tilde{\gamma}(t)},$ which concludes that $\tilde{\gamma}$ is an integral curve of $\nabla u.$
\end{proof}

\section{Proof of the Main Results}
Let $(\R^2,g)$ be the Abelian cover of the class A 2-torus and $p$ be a timelike pole on it. Observe that $T_{\mathbf{k}}(p)=p+\mathbf{k}\in\R^2$ is a timelike pole for any $\mathbf{k}\in\mathbb{Z}^{2}$, thus every future inextendible timelike geodesic $\gamma$ starting from $T_{\mathbf{k}}(p)$ is a timelike ray, i.e. $\gamma$ maximizes the $g$-arclength between any two of its points.

\begin{Pro} \label{prop-irrational}
Let $(\R^2,g)$ be the Abelian cover of the class A 2-torus $(\T^2,g)$ and $u:\R^2\rightarrow\R$ be a $C^{1,1}$ solution to the Lorentzian eikonal equation. If the integral curves of $\nabla u$ have irrational asymptotic direction, then the vector field $\nabla u$ is $\Z^2$-invariant.
\end{Pro}

\begin{proof}
From Proposition \ref{prop1}, the integral curves of $\nabla u$ are maximal. Since they possess irrational asymptotic direction, using Proposition \ref{compact_bundle} and Theorem \ref{Su4:compact-result}, we deduce that they are complete timelike lines. By Theorem \ref{Sc-2}(2), any two distinct timelike lines with the same irrational asymptotic direction are disjoint. Thus there is at most one maximal timelike line with direction $\alpha$ passing through any point of $\R^2$. Since $g$ is the lift of a Loretzian metric on $\mathbb{T}^{2}$, we concludes that $\nabla u$ is $\Z^2$-invariant.
\end{proof}

Using Proposition \ref{prop4}, we prove the existence of a $C^{1,1}$ solution to the Lorentzian eikonal equation with integral curves of a given direction.
\begin{Pro} \label{prop-main}
Let $(\mathbb{R}^2,g)$ be the Abelian cover of the class A 2-torus $(\mathbb{T}^2,g)$ having a timelike pole. Then for every asymptotic direction $\alpha\in(m^-,m^+)$, there exists a solution $\tilde{u}_{\alpha}\in C^{1,1}(\mathbb{R}^2,\mathbb{R})$ to the Equation \eqref{te} such that all integral curves of $\nabla\tilde{u}_{\alpha}$ have asymptotic direction $\alpha$.
\end{Pro}

\begin{proof}
Let $p\in\mathbb{T}^2$ be any timelike pole on $\mathbb{R}^2$. First we show that for any $\alpha\in(m^-,m^+)$, there exists a timelike ray $\gamma\in\mathscr{R}_{\alpha}$ starting from $p$. The case when $\alpha$ is rational is proved directly by Theorem \ref{exist per}. If $\alpha$ is irrational, we choose a sequence of rational $\alpha_i\in (m^-,m^+)$ such that $\alpha_i\rightarrow \alpha$ with respect to the topology defined on $(m^-,m^+)$. For each $i\in\mathbb{Z}_+$, let $\gamma_i$ be the unique periodic timelike line passing through $p=\gamma_{i}(0)$ with asymptotic direction $\alpha_i$. Thus Theorem \ref{Su4:compact-result} implies that $\{\dot{\gamma}_i(0)\}_{\mathbb{Z}_{+}}$ lies in a compact set $\text{Time}(\mathbb{R}^2,g)^{1,\delta}_{p}$ for some $\delta>0$. Hence, by taking a subsequence, $\dot{\gamma}_i(0)$ converges to some tangent vector $v\in \text{Time}(\mathbb{R}^2,g)^{1,\delta}_{p}$. Since $\gamma_i$ are geodesics, they satisfy the geodesic equations. By the continuous dependence of solutions of ODE on initial data, $\gamma_i$ converges to some timelike line $\gamma$ in the $C^0$ topology. From Theorem \ref{Sc-2}(3), the asymptotic direction is continuous w.r.t. the $C^0$ topology on the space of causal lines. Thus, $\gamma$ has asymptotic direction $\lim_{i\rightarrow \infty}\alpha_i=\alpha$.

Applying Proposition \ref{prop4} to $\gamma:[0,\infty)\rightarrow\mathbb{R}^{2}$ constructed above, we obtain a Lipschitz solution $\tilde{u}_{\alpha}\in C^{1,1}(\mathbb{R}^2,\mathbb{R})$ to Equation \eqref{te} with a sequence $\mathbf{k}_i\in\Z^{2}$ such that $T_{\mathbf{k}_i}\circ\gamma$ converges to a integral curve $\tilde{\gamma}$ of $\nabla\tilde{u}_{\alpha}$. By Proposition \ref{prop2}, $\tilde{\gamma}$ is maximal and has asymptotic direction $\alpha$. If another integral curve of $\nabla\tilde{u}_{\alpha}$ has a different asymptotic direction, they must intersect each other. This contradicts the fact that $\tilde{u}_{\alpha}$ is of $C^{1}$. Thus all integral curves of $\nabla\tilde{u}_{\alpha}$ are timelike lines with the same asymptotic direction $\alpha$.
\end{proof}

\textit{Proof of Theorem A:}

If the asymptotic direction $\alpha$ is irrational, we first apply Proposition \ref{prop-main} to construct $u_{\alpha}:=\tilde{u}_{\alpha}$. Our conclusion that $\nabla u_{\alpha}$ is $\mathbb{Z}^{2}$ periodic follows directly from Proposition \ref{prop-irrational}.

If $\alpha$ is rational, there is a sequence of irrational asymptotic direction $\alpha_i\in(m^-,m^+)$ converging to $\alpha$ as $i\rightarrow\infty$. By the conclusion in the irrational case, for each $i\in\N$, there exists a solution $u_{i}:=u_{\alpha_{i}}\in C^1(\R^2,\R)$ to Equation \eqref{te} so that
\begin{enumerate}
  \item the integral curves of $\nabla u_{i}$ have asymptotic direction $\alpha_i$,
  \item $\nabla u_{i}$ is $\mathbb{Z}^{2}$ periodic.
\end{enumerate}
By adding a constant to $u_{\alpha_{i}}$, we can suppose $u_i(x)=0$ for some fixed point $x\in\R^2$ and for all $i\in\N$. By Remark \ref{uniform Lip}, $u_{i}$ are uniformly Lipschitz with respect to the Euclidean metric. Thus there exists a subsequence $\{u_{i_{k}}\}_{k\in\N}$ converges to a function $u_{\alpha}\in C^{1}(\R^2,\R)$ in $C^{1}$ topology. This implies that $u_{\alpha}$ is a Lipschitz solution to the timelike eikonal equation \eqref{te}. We also deduce that $\nabla u_{\alpha}$ is $\mathbb{Z}^{2}$ periodic and that all integral curves of $\nabla u_{\alpha}$ have asymptotic direction $\alpha$ since these conclusions are satisfied by $u_{i}$ and closed under $C^{1}$ topology. This completes our proof.\qed

Finally, we note that
\begin{Rem}
In our former work \cite{Jin-Cui}, we prove that for any irrational direction $\alpha\in(m^{-},m^{+}),\mathscr{R}_{\alpha}$ fills $(\R^{2},g)$. Thus there is at most one global solution $u$ to Equation \eqref{te} with all integral curves of $\nabla u$ possessing direction $\alpha$. In the irrational case, $u_{\alpha}$ can be constructed as the Lorentzian Busemann function for timelike lines with direction $\alpha$.

Moreover, for the rational case, one can use Propositions \ref{prop2} and \ref{prop3} to deduce that if timelike poles appear, $b^{\pm}_{\alpha}$ (in fact $b_{\gamma}$, where $\gamma$ is any periodic timelike line starting from a timelike pole $p$) constructed in \cite{Jin-Cui} are in fact $C^{1,1}$. The conclusion there applies to our case to imply that $u_{\alpha}$ is unique when $\alpha$ is rational if and only if there is a foliation of periodic timelike lines with direction $\alpha$ on $(\R^{2},g)$.
\end{Rem}

\end{document}